\documentclass[10pt]{amsart}

\usepackage{amsmath,amsfonts, latexsym,graphicx, amssymb, mathtools, enumerate, mdwlist, amscd}

\usepackage{hyperref}
\hypersetup{colorlinks=true, urlcolor=blue, citecolor=blue, linkcolor=blue}

\newcommand{\U}{{\mathcal U}}
\newcommand{\0}{{\mathbf 0}}
\newcommand{\C}{{\mathbb C}}
\newcommand{\Z}{{\mathbb Z}}

\newcommand{\supp}{\operatorname{supp}}

\newcommand{\im}{\mathop{\rm im}\nolimits}

\newcommand{\Adot}{\mathbf A^\bullet}
\newcommand{\Bdot}{\mathbf B^\bullet}

\newcommand{\Fdot}{\mathbf F^\bullet}   

\newcommand{\Pdot}{\mathbf P^\bullet}

\newcommand{\vdual}{{\mathcal D}}

\newcommand{\coker}{{\operatorname{coker}\,}}

\newcommand{\pcoh}{{{\phantom |}^\mu\hskip -0.025in H}}
\newcommand{\var}{\mathbf{var}}
\newcommand{\can}{\mathbf{can}}

\newtheorem{defn0}{Definition}[section]
\newtheorem{prop0}[defn0]{Proposition}
\newtheorem{conj0}[defn0]{Conjecture}
\newtheorem{thm0}[defn0]{Theorem}
\newtheorem{lem0}[defn0]{Lemma}
\newtheorem{corollary0}[defn0]{Corollary}
\newtheorem{example0}[defn0]{Example}
\newtheorem{remark0}[defn0]{Remark}
\newtheorem{question0}[defn0]{Question}
\newtheorem{exercise0}[defn0]{Exercise}

\newenvironment{prop}{\begin{prop0}}{\end{prop0}}

\newenvironment{thm}{\begin{thm0}}{\end{thm0}}
\newenvironment{lem}{\begin{lem0}}{\end{lem0}}

\newenvironment{exm}{\begin{example0}\rm}{\end{example0}}
\newenvironment{rem}{\begin{remark0}\rm}{\end{remark0}}

\newcommand{\propref}[1]{Proposition~\ref{#1}}
\newcommand{\thmref}[1]{Theorem~\ref{#1}}
\newcommand{\lemref}[1]{Lemma~\ref{#1}}

\newcommand{\exref}[1]{Example~\ref{#1}}

\newcommand{\remref}[1]{Remark~\ref{#1}}

\title{A lower bound on the top betti number of the Milnor fiber}


\author{David B. Massey}

\date{}

\begin{document}

\begin{abstract} We derive a lower bound for the top possibly-non-zero betti number of the Milnor fiber of an analytic function in terms of the zero-dimensional L\^e number and the internal monodromy of the vanishing cycles restricted to the complex link of the critical locus of the function.
\end{abstract}

\maketitle




\section{Introduction}

Let $\U$ be an open subset of $\C^{n+1}$. In order to have a convenient point to discuss, we assume that $\0\in\U$. Let $f:(\U, \0)\rightarrow (\C, 0)$ be a nowhere-locally-constant complex analytic function. We let $\Sigma f$ denote the critical locus of $f$ and note that, near $\0$, this is contained in $V(f)$.

For $x\in V(f)$, we denote the Milnor fiber of $f$ at $x$ by $F_{f, x}$. We will use a field $\mathfrak K$ as our base ring for cohomology (and for perverse sheaves). The ``top'' betti number $\tilde b_n$ in the title refers to the dimension of $\widetilde H^n(F_{f,\0}; \mathfrak K)$, since $n$ is the highest degree for which the cohomology may be non-zero (note that this is a  non-traditional use of the term ``betti number'' unless $\mathfrak K$ has characteristic zero). 

\bigskip

We will use the derived category and perverse sheaves; references are  \cite{boreletal}, \cite{kashsch}, \cite{dimcasheaves}, and \cite{schurbook}. We need to use field coefficients since our proof uses a perverse sheaf version of the fact the dimension of the cokernel of a linear transformation between finite-dimensional vector spaces is equal to the dimension of the kernel of the dual transformation.

\bigskip

Let $\Pdot$ denote the shifted vanishing cycles $\phi_f[-1]{\mathfrak K}_\U^\bullet[n+1]$ along $f$. Thus $\tilde b_n=\dim H^0(\Pdot)_\0$. Let $L$ be the restriction to $V(f)$ of a generic linear form on $\C^{n+1}$; we need for $L$ to be generic enough so that $\0$ is an isolated point in the support of $\phi_L[-1]\Pdot$. Then the stalk cohomology at the origin is zero outside of degree zero;  we define $\lambda^0:=\lambda^0_f(\0)=\dim H^0(\phi_L[-1]\Pdot)_\0$. 

If $L$ is sufficiently generic, $\lambda^0$ is the $0$-th L\^e number, as defined algebraically in \cite{lecycles} as the intersection number
$$
\lambda^0=\left(\Gamma^1_{f, L}\cdot V\left(\frac{\partial f}{\partial L}\right)\right)_\0,
$$
where $\Gamma^1_{f, L}$ is the relative polar curve of Hamm, L\^e, and Teissier from \cite{hammlezariski} and \cite{teissiercargese}. Note that, while $\tilde b_n$ can change with the base field, $\lambda^0$ is independent of the field $\mathfrak K$.

\bigskip

In \thmref{thm:main}, we give a lower bound -  in terms of $\lambda^0$ and the internal monodromy - for $\tilde b_{n}$, for an $f$ where the critical locus has arbitrary dimension. $T^L_{\Pdot}$ is the $L$-monodromy automorphism on $\psi_L[-1]\phi_f[-1]\C_\U^\bullet[n+1]$ (letting the value of $L$, not $f$, go around a small circle). We denote the image of $\operatorname{id}-T^L_{\Pdot}$ (in the abelian category of perverse sheaves) by $\im\{\operatorname{id}-T^L_{\Pdot}\}$; note that, as we shall discuss later, $\0$ is an isolated point in the support of this image.

\medskip

The first version of our main theorem is:

\smallskip

\begin{thm}\label{thm:mainintro} Regardless of the dimension of $\Sigma f$,
$$
\frac{\lambda^0-\dim\big(H^0(\im\{\operatorname{id}-T^L_{\Pdot}\})_\0\big)}{2}\ \leq \ \tilde b_n\  \leq \ \lambda^0-\dim\big(H^0(\im\{\operatorname{id}-T^L_{\Pdot}\})_\0\big).
$$

In particular, if $T^L_{\Pdot}=\operatorname{id}$, then $
\lambda^0/2\leq \tilde b_n\leq \lambda^0$.
\end{thm}

\smallskip

 In Section 2, we prove several lemmas that we need. In Section 3, we prove the main theorem,  \thmref{thm:main}, and give an example of the well-studied case where $\dim_\0\Sigma f=1$. In Section 4, we analyze the case where $\dim_\0\Sigma f>1$ more carefully and present \thmref{thm:mainv2}, the second version of our main theorem, in which we make the calculation of  $H^0(\im\{\operatorname{id}-T^L_{\Pdot}\})_\0$ more manageable.

\medskip

\section{Lemmas}

We continue with our base ring being the field $\mathfrak K$. Let $X$ be a complex analytic space. The category $\operatorname{Perv}(X)$ of perverse sheaves on $X$ is abelian.

\begin{lem}\label{lem:stalks} Let $x\in X$ and suppose that $x$ is an isolated point in the support of $\Pdot\in \operatorname{Perv}(X)$, which implies that the stalk cohomology $H^k(\Pdot)_x$ is zero if $k\neq 0$.

 Suppose we have maps of perverse sheaves $\Adot\xrightarrow{\alpha}\Pdot$ and $\Pdot\xrightarrow{\beta}\Bdot$. Let $\alpha_x^k$ and $\beta_x^k$ denote the corresponding maps on stalk cohomology:
 $$\alpha^k_x: H^k(\Adot)_x\rightarrow  H^k(\Pdot)_x\hskip 0.2in\textnormal{ and }\hskip 0.2in \beta^k_x: H^k(\Pdot)_x\rightarrow  H^k(\Bdot)_x.
 $$
 Then,
 \begin{enumerate}
 \item $x$ is an isolated point in the supports of $\im\alpha$, $\coker\alpha$, $\im\beta$, $\ker\beta$, and $\im(\beta\alpha)$;
 \medskip
 \item $H^0(\im\alpha)_x\cong \im(\alpha^0_x)$ and $H^0(\coker\alpha)_x\cong \coker(\alpha^0_x)$; and
 \medskip
 \item $H^0(\im\beta)_x\cong \im(\beta^0_x)$ and $H^0(\ker\beta)_x\cong \ker(\beta^0_x)$.
 \end{enumerate}
\end{lem}
\begin{proof} We will prove the statements regarding $\alpha$. The proofs of the statements regarding $\beta$ are completely analogous; we leave them as an exercise. That $x$ is isolated in $\operatorname{supp}\im(\beta\alpha)$ is immediate from $x$ being isolated in $\operatorname{supp}\im\alpha$ (or in $\operatorname{supp}\im(\beta\alpha)$).

We have the canonical exact sequence in $\operatorname{Perv}(X)$:
$$
0\rightarrow \ker\alpha\rightarrow \Adot\xrightarrow{\alpha}\Pdot\rightarrow\coker\alpha,
$$
which we can split into two short exact sequences
\begin{equation}\label{eq:twoshort}
0\rightarrow \ker\alpha\rightarrow \Adot\xrightarrow{\sigma}\im\alpha\rightarrow 0\hskip 0.2in\textnormal{ and }\hskip 0.2in 0\rightarrow \im\alpha\xrightarrow{\tau} \Pdot\rightarrow\coker\alpha\rightarrow 0,
\end{equation}
where $\alpha=\tau\circ\sigma$. From the second short exact sequence above, if we restrict to $\U-\{x\}$ where $\U$ is a small open neighborhood of $x$, we obtain $0$ in the middle; this shows that $x$ is an isolated point in the supports of $\im\alpha$ and $\coker\alpha$. From this, we immediately conclude that we have a short exact sequence
$$
0\rightarrow H^0(\im\alpha)_x\xrightarrow{\tau^0_x} H^0(\Pdot)_x\rightarrow H^0(\coker\alpha)_x\rightarrow 0.
$$

Now the first short exact sequence in \eqref{eq:twoshort} gives us the following short exact sequence
$$
0\rightarrow H^0(\ker\alpha)_x\rightarrow H^0(\Adot)_x\xrightarrow{\sigma^0_x}H^0(\im\alpha)_x\rightarrow 0.
$$
Combining these two last short exact sequences and using that $\alpha^0_x=\tau^0_x\circ\sigma^0_x$, we conclude the isomorphisms in Item 2.

\end{proof}

\medskip

\begin{rem}\label{rem:warning} The reader should appreciate that there really was something to prove in \lemref{lem:stalks}; some things which may seem ``obvious'' are, in fact, not true. Suppose we have a map $\Adot\xrightarrow{\gamma}\Bdot$ between perverse sheaves on $X$  such that a point $x\in X$ is isolated in the support of $\im\gamma$. Then, in general, $H^0(\im \gamma)_x\not\cong \im(\gamma^0_x)$. See Example 3.3 in \cite{kernelnote}.
\end{rem}

\bigskip

Below, we use the notation from \lemref{lem:stalks}.

\begin{lem}\label{lem:inequality} Let $\Adot$ and $\Pdot$ be perverse sheaves on $X$. Let $x\in X$ and suppose that $x$ is an isolated point in the support of $\Pdot$. Suppose that we have maps $\Adot\xrightarrow{\alpha}\Pdot\xrightarrow{\beta}\Adot$ such that $H^0(\coker \alpha)_x\cong H^0(\ker\beta)_x$. 

Then,
\begin{enumerate}
\item
$\dim(\coker(\alpha^0_x))=\dim H^0(\coker\alpha)_x,
$

\medskip

\item

$x$ is an isolated point in the support of $\im(\beta\alpha)$, and

\medskip

\item
$$
 \frac{\dim H^0(\Pdot)_x-\dim H^0(\im(\beta\alpha))_x}{2}\leq\dim(\coker(\alpha^0_x))\leq\dim H^0(\Pdot)_x-\dim H^0(\im(\beta\alpha))_x.
$$
\end{enumerate}
\end{lem}
\begin{proof} Items 1 and 2 follow from \lemref{lem:stalks}.

\smallskip

In the remainder of this proof, to eliminate cumbersome notation, when $x$ is an isolated point in the support of some perverse sheaf $\Fdot$, we shall write simply $\dim \Fdot$ in place of $\dim H^0(\Fdot)_x$.

\medskip

The proof is really just an exercise in linear algebra. 

\medskip

The second inequality is the most simple.
$$
\dim (\im(\beta\alpha))\leq \dim (\im\beta) = \dim \Pdot-\dim(\ker\beta)=\dim\Pdot-\dim(\coker\alpha).
$$

\medskip

For the first inequality, we have:
$$\im(\beta\alpha) =\im(\beta_{|_{\im\alpha}})\cong\frac{\im\alpha}{\ker\beta\cap\im\alpha}.
$$

$$
\dim(\im(\beta\alpha))=\dim(\im\alpha)-\dim(\ker\beta\cap\im\alpha)\geq \dim(\im\alpha)-\dim(\ker\beta).
$$

\medskip

So, using that $\dim\coker \alpha\cong\dim\ker\beta$,
$$
\dim(\im(\beta\alpha))\geq \dim(\im\alpha)-\dim(\coker\alpha).
$$

$$
-\dim(\im(\beta\alpha))\leq -\dim(\im\alpha)+\dim(\coker\alpha).
$$

$$
\dim \Pdot-\dim(\im(\beta\alpha))\leq \dim\Pdot-\dim(\im\alpha)+\dim(\coker\alpha)=2\dim(\coker\alpha).
$$

$$
\dim(\coker\alpha)\geq \frac{\dim \Pdot-\dim(\im(\beta\alpha))}{2}.
$$
\end{proof}

\medskip

\begin{rem} Recall our warning from \remref{rem:warning}. It could easily be that $H^0(\im(\beta\alpha))_x$ is {\bf not} isomorphic to $\im(\beta^0_x\circ\alpha^0_x)$.
\end{rem}

\bigskip

Before we give our final lemma of this section, we will need to have a long discussion, which will include the proof.

\bigskip

Since our base ring is a field, Verdier dualizing takes perverse sheaves to perverse sheaves and so takes  exact sequences of perverse sheaves to exact sequences of perverse sheaves.

\medskip

Let $\Pdot$ be a perverse sheaf on $X$ and let $g:X\rightarrow\C$ be a complex analytic function such that some point $x\in X$ is an isolated point in $\operatorname{supp}\phi_g[-1]\Pdot$. Then the stalk cohomology $H^*(\phi_g[-1]\Pdot)_x$ is zero except, possibly, in degree zero.

Let $m: V(g)\hookrightarrow X$ be the inclusion. Consider the two standard natural distinguished triangles
\begin{equation}\label{eq:cantri}
\rightarrow m_*m^*[-1]\rightarrow\psi_g[-1]\xrightarrow{\can^g}\phi_g[-1]\xrightarrow{[1]}
\end{equation}
and
$$
\rightarrow\phi_g[-1]\xrightarrow{\var^g}\psi_g[-1]\rightarrow m_!m^![1]\xrightarrow{[1]}.
$$

\medskip

Applying these to $\Pdot$ and taking perverse cohomology, we have  exact sequences of perverse sheaves:

$$
0\rightarrow\pcoh^0(m_*m^*[-1]\Pdot)\rightarrow\psi_g[-1]\Pdot\xrightarrow{\can^g_{\Pdot}}\phi_g[-1]\Pdot\rightarrow \pcoh^1(m_*m^*[-1]\Pdot)\rightarrow 0
$$
and
$$
0\rightarrow\pcoh^{-1}(m_!m^![1]\Pdot)\rightarrow\phi_g[-1]\Pdot\xrightarrow{\var^g_{\Pdot}}\psi_g[-1]\Pdot\rightarrow \pcoh^{0}(m_!m^![1]\Pdot)\rightarrow 0.
$$

\vskip 0.2in

\noindent Thus, $\coker\{\can^g_{\Pdot}\}\cong \pcoh^1(m_*m^*[-1]\Pdot)$ and $\ker\{\var^g_{\Pdot}\}\cong \pcoh^{-1}(m_!m^![1]\Pdot)$.

\bigskip

Note that $x$ is an isolated point in the supports of both of these perverse sheaves; so, at $x$, they are effectively just finite-dimensional vector spaces in degree $0$. 

\medskip

The composition $\var^g_{\Pdot}\circ\can^g_{\Pdot}$ is equal to $\operatorname{id}-T^g_{\Pdot}$, where   $T^g_{\Pdot}$ is the monodromy automorphism on $\psi_g[-1]\Pdot$.

\medskip

Now suppose that  $\Pdot$ be a self-dual perverse sheaf on $X$; that is, suppose there is an isomorphism $\Pdot\xrightarrow[\cong]{\omega}\vdual \Pdot$, where $\vdual$ is the Verdier dualizing functor. As we are using field coefficents, dualizing commutes with perverse cohomology with a negation in the degree. Thus,
$$
\vdual\big(\coker\{\can^g_{\Pdot}\}\big)\cong \vdual\pcoh^1(m_*m^*[-1]\Pdot)\cong \pcoh^{-1}(\vdual m_*m^*[-1]\Pdot)\cong
$$
$$
 \pcoh^{-1}(m_!m^![1]\vdual\Pdot)\cong \pcoh^{-1}(m_!m^![1]\Pdot)\cong \ker\{\var^g_{\Pdot}\}.
$$ 

\bigskip

 Note that taking stalk cohomology in the distinguished triangle \eqref{eq:cantri} tells us that
\begin{equation}\label{eq:tired}
\coker\{(\can^g_{\Pdot})^0_x\}\cong H^0(\Pdot)_x.
\end{equation}

\bigskip

\noindent From this discussion, and \lemref{lem:stalks}, it  follows immediately that we have:

\medskip
\begin{lem}\label{lem:selfdual} Let $\Pdot$ be a perverse sheaf on $X$. Suppose that $g:X\rightarrow\C$ is a complex analytic function such that a point $x\in X$ is an isolated point in $\operatorname{supp}\phi_g[-1]\Pdot$. Then, 

\begin{enumerate}
\item 
$x$ is an isolated point in $\supp \coker\{\can^g_{\Pdot}\}$, in $\supp\ker\{\var^g_{\Pdot}\}$,  in $\supp\im\{\var^g_{\Pdot}\}$ and in $\supp\im\{\operatorname{id}-T^g_{\Pdot}\}$;
\medskip
\item
the stalk cohomologies at $x$ of $\phi_g[-1]\Pdot$, $\coker\{\can^g_{\Pdot}\}$,  $\ker\{\var^g_{\Pdot}\}$, and $\im\{\operatorname{id}-T^g_{\Pdot}\}$ are all concentrated in degree 0; and
\medskip
\item
if $\Pdot$ is a self-dual perverse sheaf, we have
$$\dim H^0(\Pdot)_x= \dim\coker\{(\can^g_{\Pdot})^0_x\}=\dim H^0(\coker\{\can^g_{\Pdot}\})_x=
$$
$$ \dim H^0(\ker\{\var^g_{\Pdot}\})_x = \dim\ker\{(\var^g_{\Pdot})^0_x\}.$$
\end{enumerate}
\end{lem}

\medskip

\section{The Main Theorem} 

We return to the setting of the introduction. Let $\U$ be an open subset of $\C^{n+1}$, let $\0\in\U$, and let $f:(\U, \0)\rightarrow (\C, 0)$ be a nowhere-locally-constant complex analytic function.  

\medskip

Let $F_{f, \0}$ denote the Milnor fiber of $f$ at $\0$. Let $\tilde b_{n}:=\dim \widetilde H^{n}(F_{f, \0}; \mathfrak K)$.

\medskip

Let $\Pdot:=\phi_f[-1]{\mathfrak K}_\U^\bullet[n+1]$. As is well-known, $\Pdot\cong \vdual\Pdot$ (see, for instance, \cite{natcommute} and use that $\vdual{\mathfrak K}_\U^\bullet[n+1]\cong {\mathfrak K}_\U^\bullet[n+1]$). As in the introduction, let $L$ be the restriction to $V(f)$ of a generic linear form on $\C^{n+1}$; we need for $L$ to be generic enough so that $\0$ is an isolated point in the support of $\phi_L[-1]\Pdot$ (in \cite{lecycles}, the condition is that $L$ is {\it prepolar}). Let $\lambda^0:=\dim H^0(\phi_L[-1]\Pdot)_\0$ . This $L$ will take the place of $g$ in \lemref{lem:selfdual}.

Referring to the previous section, we have the maps
$$
\psi_L[-1]\Pdot\xrightarrow{\can^L_{\Pdot}}\phi_L[-1]\Pdot\xrightarrow{\var^L_{\Pdot}}\psi_L[-1]\Pdot.
$$

\medskip

\noindent The composition $\var^L_{\Pdot}\circ\can^L_{\Pdot}$ is equal to $\operatorname{id}-T^L_{\Pdot}$, where   $T^L_{\Pdot}$ is the monodromy automorphism on $\psi_L[-1]\Pdot$ (letting the value of $L$, not $f$, go around a small circle).

\bigskip

Now, given the lemmas of the previous section, the proof of the main theorem will be quite short.

\begin{thm}\label{thm:main}\textnormal{(\bf{Main Theorem, version 1})} Regardless of the dimension of $\Sigma f$, $\0$ is isolated in the support of $\im\{\operatorname{id}-T^L_{\Pdot}\}$ and

\medskip

$$
\frac{\lambda^0-\dim H^0(\im\{\operatorname{id}-T^L_{\Pdot}\})_\0}{2}\ \leq \ \tilde b_n\  \leq \ \lambda^0-\dim H^0(\im\{\operatorname{id}-T^L_{\Pdot}\})_\0.
$$

\medskip

In particular, if $T^L_{\Pdot}=\operatorname{id}$, then $
\lambda^0/2\leq \tilde b_n\leq \lambda^0$.

\end{thm}
\begin{proof} \eqref{eq:tired} tells us that 
$$
\tilde b_n=\dim H^0(\Pdot)_\0= \dim\coker\{(\can^L_{\Pdot})^0_\0\} .
$$

\medskip

\lemref{lem:selfdual} tells us that $\0$ is isolated in the support of $\im\{\operatorname{id}-T^L_{\Pdot}\}$ and that 
$$\dim H^0(\coker\{\can^L_{\Pdot}\})_\0=\dim H^0(\ker\{\var^L_{\Pdot}\})_\0 .$$

Thus, the hypotheses of \lemref{lem:inequality} are satisfied, with $\alpha=\can^L_{\Pdot}$ and $\beta=\var^L_{\Pdot}$.  and we conclude that
$$
\tilde b_n=\dim\coker\{(\can^L_{\Pdot})^0_\0\}=\dim H^0(\coker\{\can^L_{\Pdot}\})_\0
$$

and

$$
\frac{\lambda^0-\dim H^0(\im\{\operatorname{id}-T^L_{\Pdot}\})_\0}{2}\ \leq \ \tilde b_n\  \leq \ \lambda^0-\dim H^0(\im\{\operatorname{id}-T^L_{\Pdot}\})_\0.
$$

\end{proof}

\bigskip

\begin{rem} One may wish to use a finite field as the base ring in order to have a statement which involves torsion in the integral cohomology in the Milnor fiber. In addition, in theory, 
$$\dim H^0(\im\{\operatorname{id}-T^L_{\Pdot}\})_\0$$
 could change when we change base fields.

\medskip

As we mentioned in the introduction, if $L$ is sufficiently generic, $\lambda^0$ is the $0$-th L\^e number, as defined algebraically in \cite{lecycles} as the intersection number
$$
\lambda^0=\left(\Gamma^1_{f, L}\cdot V\left(\frac{\partial f}{\partial L}\right)\right)_\0,
$$
and thus is independent of the base field.

\end{rem}

\begin{exm}\label{exm:1dim}
Suppose that $\dim_\0\Sigma f=1$. In this case, we can relate \thmref{thm:main} to a good deal of previous work in the area.

Shrinking $\U$ if necessary, we assume that each component $C$ of $\Sigma f$ is contained in $V(f)$, that $C-\{\0\}$ is a Whitney stratum of $V(f)$, and that $C-\{\0\}$ is homeomorphic to a punctured disk. If one takes a generic transverse hyperplane slice $H$ at a point $p\in C-\{\0\}$, then the isomorphism-type of the degree $(n-1)$ cohomology of the Milnor fiber $F_{f_{|_H}, p}$ is independent of $p\in C-{\0}$ and is ${\mathfrak K}^{\mu^\circ_C}$, where $\mu^\circ_C$ is called the generic transverse Milnor number on $C$.

In addition, since each $C-\{\0\}$ is homeomorphic to a punctured disk, the fundamental group of $C-\{\0\}$ is isomorphic to $\Z$. And so, corresponding to either one of the two generators of $\pi_1(C-\{\0\})$, for each $C$, we obtain an {\it internal (a.k.a. vertical) vanishing cycle  monodromy isomorphism} 
$$h_C: {\mathfrak K}^{\mu^\circ_C}\xrightarrow[\cong]{} {\mathfrak K}^{\mu^\circ_C}$$ by traveling around the generator. This should {\bf not} be confused with the Milnor monodromy, where the value of $f$ travels around a circle.

It is a well-known result of Siersma \cite{siersmavarlad} that there is an injection
\begin{equation}\label{eq:siersmainj}
\widetilde H^{n-1}(F_{f, \0}; {\mathfrak K})\hookrightarrow \bigoplus_C\ker\{\operatorname{id}-h_C\}.
\end{equation}
Thus, in a sense, the most problematic case for obtaining a ``nice'' upper-bound on the (reduced) degree $(n-1)$ betti number of $F_{f, \0}$ is the case where all of the $h_C$ are the identity.

Now, for a generic linear form $L$, there are the algebraically defined L\^e numbers of $f$ at $\0$:
$$
\lambda^0:=\lambda^0_{f, L}(\0) \hskip 0.2in \textnormal{ and }\hskip 0.2in \lambda^1:=\lambda^1_{f, L}(\0);
$$
see \cite{lecycles}. The 1-dimensional L\^e number is easy to describe: 
$$
\lambda^1=\sum_C (\operatorname{mult} C) \mu^\circ_C=\dim_\0 H^0(\psi_L[-1]\phi_f[-1]{\mathfrak K}_\U^\bullet[n+1])_\0,
$$
where we do not distinguish in the notation between $L$ and its restriction to $V(f)$.

The definition of the 0-dimensional L\^e number is the intersection number
$$
\lambda^0=\left(\Gamma^1_{f, L}\cdot V\left(\frac{\partial f}{\partial L}\right)\right)_\0= \dim_\0 H^0(\phi_L[-1]\phi_f[-1]{\mathfrak K}_\U^\bullet[n+1])_\0.
$$
Let us denote the betti numbers of the Milnor fiber by
$$
\tilde b_{n-1}:=\dim \widetilde H^{n-1}(F_{f, \0}; \mathfrak K)\hskip 0.2in\textnormal{ and }\hskip 0.2in \tilde b_{n}:=\dim \widetilde H^{n}(F_{f, \0}; \mathfrak K).
$$
As part of the general theory of L\^e cycles and numbers, we have
$$
\tilde b_{n-1}\leq\lambda^1,\hskip 0.2in \tilde b_{n}\leq\lambda^0, \hskip 0.2in\textnormal{ and }\hskip 0.2in \tilde b_n-\tilde b_{n-1}= \lambda^0-\lambda^1.
$$
Because of the equality above, bounds on $\tilde b_{n-1}$ yield bounds on $\tilde b_{n}$ (and vice versa). Note that \eqref{eq:siersmainj} tells us that the only way that we could possibly have $\tilde b_{n-1}=\lambda^1$ is for each component $C$ of $\Sigma f$ to be smooth at the origin and for the internal monodromy to be trivial. 

So, suppose all components $C$ of $\Sigma f$ are smooth at the origin and that, for all $C$, the internal monodromy is trivial, i.e., $h_C=\operatorname{id}$ for all $C$. Then is it possible that $\tilde b_{n-1}=\lambda^1$ or, equivalently, that $\tilde b_{n}=\lambda^0$? {\bf No}, unless $\lambda^0=0$; see \cite{lemassey} in the general case and \cite{siersmaisoline} in the case where $\lambda^1=1$. 

\medskip

How does any of this relate to \thmref{thm:main}?

\medskip

With $\Pdot:=\phi_f[-1]{\mathfrak K}_\U^\bullet[n+1]$ and $T^L_{\Pdot}$ being the $L$-monodromy automorphism on $\psi_L[-1]\Pdot$, it is an easy exercise to show that 
$$
\bigoplus_C\ker\{\operatorname{id}-h_C\} \ \cong \ \ker\{(\operatorname{id}-T^L_{\Pdot})^0_\0\}\cong H^0\big(\ker\{\operatorname{id}-T^L_{\Pdot}\}\big)_\0.
$$

Therefore, the injection \eqref{eq:siersmainj} tells us that
$$
\tilde b_{n-1}\leq \dim H^0\big(\ker\{\operatorname{id}-T^L_{\Pdot}\}\big)_\0
$$ 
or, equivalently, that
$$
\tilde b_{n-1}\leq \lambda^1- \dim H^0\big(\im\{\operatorname{id}-T^L_{\Pdot}\}\big)_\0.
$$
Since $\lambda^0-\lambda^1=\tilde b_n-\tilde b_{n-1}$, the inequalities above are equivalent to 
$$
\tilde b_n\leq \lambda^0-\dim H^0\big(\im\{\operatorname{id}-T^L_{\Pdot}\}\big)_\0;
$$
this is the upper bound in \thmref{thm:main}.

\medskip

However, as we wrote in the introduction, the lower bound in \thmref{thm:main} really seems to be new.
\end{exm}

\medskip

\section{The case of higher-dimensional critical loci}

Throughout this section, we continue with all of our previous notation; in particular, $\Pdot$ is the perverse sheaf $\phi_f[-1]{\mathfrak K}_\U^\bullet[n+1]$.

\medskip

In the previous section, we saw that the upper bound in \thmref{thm:main} is equivalent to a well-known upper bound in the case where $\dim_\0\Sigma f=1$. However, in this section, we assume $\dim_\0\Sigma f>1$.

With this assumption, both bounds on the top betti number in \thmref{thm:main} are new. However, it is hard to work with $\dim H^0\big(\im\{\operatorname{id}-T^L_{\Pdot}\}\big)_\0$ since it need not equal the dimension of the image of the map on stalk cohomology. In this section, we will ``fix'' this problem.

\medskip

We will assume that, after a generic linear change of coordinates (generic with respect to $f$ at the origin), that we have prepolar coordinates $\mathbf z:=(z_0, z_1, \dots, z_n)$ for $\C^{n+1}$; prepolar coordinates are defined and used throughout \cite{lecycles}. In the remainder of this section, we shall not distinguish in the notation between the coordinate functions $z_k$ and their restrictions to subspaces. The function $z_0$ will be used in place of the linear form $L$ from the previous sections.

\medskip

As the coordinates $\mathbf z$ are prepolar, the L\^e numbers of $f$ at the origin exist and are defined algebraically in \cite{lecycles}. As we wrote earlier, the definition of the 0-dimensional L\^e number is the intersection number
$$
\lambda^0=\lambda^0_{f, \mathbf z}(\0):=\left(\Gamma^1_{f, z_0}\cdot V\left(\frac{\partial f}{\partial z_0}\right)\right)_\0.
$$
However, the easy characterization of $\lambda^1=\lambda^1_{f, \mathbf z}(\0)$ that we gave in \exref{exm:1dim} does not apply in the case where $\dim_\0\Sigma f>1$. We shall not give the definition here, but will show how it is calculated in the midst of \exref{exm:2dim} below.

What is important here is that, for prepolar coordinates, the algebraic definition agrees with the calculation via iterated nearby and vanishing cycles. That is, Theorem 10.8 of \cite{lecycles} or  Theorem 7.6.9 of \cite{lesurvey} give us that $\0$ is an isolated point in the support of each of the perverse sheaves $\phi_{z_0}[-1]\Pdot$ and $\phi_{z_1}[-1]\psi_{z_0}[-1]\Pdot$, and

$$
\lambda^0_{f, \mathbf z}(\0)=\dim H^0\big(\phi_{z_0}[-1]\Pdot\big)_\0
$$
and
$$
\lambda^1_{f, \mathbf z}(\0)=\dim H^0\big(\phi_{z_1}[-1]\psi_{z_0}[-1]\Pdot\big)_\0.
$$

\bigskip

Now, we want to describe $H^0(\im\{\operatorname{id}-T^{z_0}_{\Pdot}\})_\0$ as the image of a map on stalk cohomology. We prove:

\begin{prop}\label{prop:one} Consider the isomorphism
$$
\phi_{z_1}[-1](T^{z_0}_{\Pdot}): \phi_{z_1}[-1]\psi_{z_0}[-1]\Pdot\rightarrow \phi_{z_1}[-1]\psi_{z_0}[-1]\Pdot
$$
and its induced map on stalk cohomology
$$
\big(\phi_{z_1}[-1](T^{z_0}_{\Pdot})\big)^0_\0: H^0\big(\phi_{z_1}[-1]\psi_{z_0}[-1]\Pdot\big)_\0\rightarrow H^0\big(\phi_{z_1}[-1]\psi_{z_0}[-1]\Pdot\big)_\0.$$

There is an isomorphism
$$
H^0(\im\{\operatorname{id}-T^{z_0}_{\Pdot}\})_\0 \ \cong \ \im\left\{\big(\operatorname{id}-\phi_{z_1}[-1](T^{z_0}_{\Pdot})\big)^0_\0\right\}.
$$
\end{prop}
\begin{proof} First, note that, since $\0$ is isolated in the support of $\im\{\operatorname{id}-T^{z_0}_{\Pdot}\}$ (as we saw in the previous section),
$$
H^0\big(\im\{\operatorname{id}-T^{z_0}_{\Pdot}\}\big)_\0\cong H^0\big(\phi_{z_1}[-1](\im\{\operatorname{id}-T^{z_0}_{\Pdot}\})\big)_\0.
$$

\medskip

\noindent So we need to show
\begin{equation}\label{eq:last}
H^0\big(\phi_{z_1}[-1](\im\{\operatorname{id}-T^{z_0}_{\Pdot}\})\big)_\0\ \cong \ \im\left\{\big(\operatorname{id}-\phi_{z_1}[-1](T^{z_0}_{\Pdot})\big)^0_\0\right\}.
\end{equation}

\medskip

The argument is entirely similar to that in \lemref{lem:stalks}. Consider the exact sequence of perverse sheaves
$$
0\rightarrow\ker\{\operatorname{id}-T^{z_0}_{\Pdot}\}\rightarrow \psi_{z_0}[-1]\Pdot\xrightarrow{\operatorname{id}-T^{z_0}_{\Pdot}}\psi_{z_0}[-1]\Pdot\rightarrow \coker\{\operatorname{id}-T^{z_0}_{\Pdot}\}\rightarrow 0.
$$
Split this into two short exact sequences
$$
0\rightarrow\ker\{\operatorname{id}-T^{z_0}_{\Pdot}\}\rightarrow \psi_{z_0}[-1]\Pdot\xrightarrow{\alpha}\im\{\operatorname{id}-T^{z_0}_{\Pdot}\}\rightarrow 0
$$
and
$$
0\rightarrow\im\{\operatorname{id}-T^{z_0}_{\Pdot}\}\xrightarrow{\beta} \psi_{z_0}[-1]\Pdot\rightarrow\coker\{\operatorname{id}-T^{z_0}_{\Pdot}\}\rightarrow 0,
$$
where $\operatorname{id}-T^{z_0}_{\Pdot}=\beta\alpha$.

Now apply the functor $\phi_{z_1}[-1]$ to both of these short exact sequences. We obtain new short exact sequences:

$$
0\rightarrow\phi_{z_1}[-1]\big(\ker\{\operatorname{id}-T^{z_0}_{\Pdot}\}\big)\rightarrow \phi_{z_1}[-1]\psi_{z_0}[-1]\Pdot\xrightarrow{\phi_{z_1}[-1](\alpha)}\phi_{z_1}[-1]\big(\im\{\operatorname{id}-T^{z_0}_{\Pdot}\})\rightarrow 0
$$
and
$$
0\rightarrow\phi_{z_1}[-1]\big(\im\{\operatorname{id}-T^{z_0}_{\Pdot}\}\big)\xrightarrow{\phi_{z_1}[-1](\beta)} \phi_{z_1}[-1]\psi_{z_0}[-1]\Pdot\rightarrow\phi_{z_1}[-1]\big(\coker\{\operatorname{id}-T^{z_0}_{\Pdot}\}\big)\rightarrow 0,
$$

\smallskip

\noindent where $\operatorname{id}-\phi_{z_1}[-1](T^{z_0}_{\Pdot})=\phi_{z_1}[-1](\beta)\circ\phi_{z_1}[-1](\alpha)$ and $\0$ is an isolated point in the support of all of the perverse sheaves appearing. Thus, these two short exact sequences yield two short exact sequences on the degree $0$ stalk cohomologies at $\0$. From this, one easily concludes that \eqref{eq:last} holds, and we are finished.
\end{proof}

\medskip

Combining \propref{prop:one} with  \thmref{thm:main}, we have:

\smallskip

\begin{thm}\label{thm:mainv2}\textnormal{(\bf{Main Theorem, version 2})}

The following inequalities hold:

\medskip

$$
\frac{\lambda^0-\dim \left(\im\left\{\big(\operatorname{id}-\phi_{z_1}[-1](T^{z_0}_{\Pdot})\big)^0_\0\right\}\right)}{2}\ \leq \ \tilde b_n\  \leq \ \lambda^0-\dim \left(\im\left\{\big(\operatorname{id}-\phi_{z_1}[-1](T^{z_0}_{\Pdot})\big)^0_\0\right\}\right),
$$

or, equivalently, 

$$
 \lambda^0-2\tilde b_n\ \leq \ \dim \left(\im\left\{\big(\operatorname{id}-\phi_{z_1}[-1](T^{z_0}_{\Pdot})\big)^0_\0\right\}\right)\  \leq \ \lambda^0-\tilde b_n.
$$

\medskip

In particular, $\tilde b_n=0$ if and only if $\lambda^0=\dim \left(\im\left\{\big(\operatorname{id}-\phi_{z_1}[-1](T^{z_0}_{\Pdot})\big)^0_\0\right\}\right)$.

\end{thm}

\bigskip

\begin{exm}\label{exm:2dim} Consider $f:\C^{4}\rightarrow\C$ given by $f(u,v,x,y):=y^2+x^5+ux^4+v^2x^2$. We use the coordinate system $(u, v, x, y)$ (in that order). We will calculate the L\^e cycles $\Lambda^k$ and L\^e numbers $\lambda^k$ at $\0$; along the way, our calculations will show that these coordinates are prepolar. As part of the process, we must calculate the relative polar cycles $\Gamma^k$. We will suppress references to $f$ and the coordinate system for the L\^e cycles/numbers and polar cycles.

We have, as sets,

$$
\Sigma f= V\left(\frac{\partial f}{\partial u}, \frac{\partial f}{\partial v}, \frac{\partial f}{\partial x}, \frac{\partial f}{\partial y} \right) = V(x^4, 2vx^2, 5x^4+4ux^3+2xv^2, 2y) = V(x,y).
$$

We proceed as in \cite{lecycles} to find that, as cycles:
$$
\Gamma^3=V(2y)=V(y).
$$

\smallskip

$$
\Gamma^3\cdot V(5x^4+4ux^3+2xv^2)=V(y)\cdot \big[V(x)+V(5x^3+4ux^2+2v^2)\big] = 
$$
$$
V(5x^3+4ux^2+2v^2, y)+V(x,y)= \Gamma^2 + \Lambda^2 \ \ \textnormal{(respectively)}.
$$

\smallskip

$$
\Gamma^2\cdot V(2vx^2)= V(5x^3+4ux^2+2v^2, y)\cdot \big[V(v)+2V(x)\big]=
$$
$$
V(5x^3+4ux^2, v,y)+2V(5x^3+4ux^2+2v^2, x,y)= V(5x+4u, v,y)+2V(x,y,v)+ 4V(x,y,v)=
$$
$$
V(5x+4u, v,y)+6V(x,y,v)= \Gamma^1 + \Lambda^1.
$$

\smallskip

$$
\Gamma^1\cdot V(x^4)= V(5x+4u, v,y)\cdot 4V(x)=4[\0]=\Lambda^0.
$$

\medskip

Thus, $\lambda^2=1$, $\lambda^1=6$, and $\lambda^0=4$. This means there is a chain complex
$$
0\rightarrow {\mathfrak K}^1\rightarrow {\mathfrak K}^6\rightarrow {\mathfrak K}^4\rightarrow 0,
$$
where the cohomology in the ${\mathfrak K}^{\lambda^k}$ position is isomorphic to $\widetilde H^{3-k}(F_{f, \0}; {\mathfrak K})$.

\medskip

But for this $f$, we can calculate the cohomology precisely. Let $g=x^5+ux^4+v^2x^2$. Then $f$ is referred to as the suspension of $g$, for the Milnor fiber of $f$ at $\0$ has the homotopy-type of the suspension of the Milnor fiber of $g$ at $\0$ and, hence, the reduced cohomology of $F_{f, \0}$ is obtained from the reduced cohomology of $g$, shifted up one degree.

Since $g$ is weighted-homogeneous, $F_{g, \0}$ is diffeomorphic to $g^{-1}(1)$.  But we can solve for $u$ in the equation $x^5+ux^4+v^2x^2=1$ to obtain
$$
u=\frac{1-x^5-v^2x^2}{x^4}.
$$
Therefore $F_{g, \0}$ is homotopy-equivalent to $S^1$, and so the reduced cohomology of $F_{f, \0}$ is zero outside of degree 2 and $\widetilde H^2(F_{f, \0};\mathfrak K)\cong \mathfrak K$. In particular $\tilde b_3=0$, and thus \thmref{thm:mainv2} tells us that
$$
4=\lambda^0=\dim \left(\im\left\{\big(\operatorname{id}-\phi_{v}[-1](T^{u}_{\Pdot})\big)^0_\0\right\}\right).
$$
\end{exm}

\medskip

\begin{rem} In \exref{exm:2dim}, we used \thmref{thm:mainv2} ``backwards'' from the way one might expect. Rather than using data about the perverse sheaf of vanishing cycles near, but not at, the origin to gain information about the Milnor fiber at the origin, we used information at the origin to tell us something about the structure of the vanishing cycles nearby.
\end{rem}

\medskip

\bibliographystyle{plain}

\bibliography{Masseybib}

\end{document}